\documentclass[12pt]{amsart} %
\usepackage{fullpage, subfigure, layout, setspace}
\usepackage{graphicx}
\usepackage{graphics}
\usepackage{amsmath}
\usepackage{amsfonts}
\usepackage{amssymb}
\usepackage{amsthm}
\usepackage{psfrag}
\usepackage{amsmath, amssymb, color}
\usepackage{amsfonts,amstext}
\usepackage{enumerate}

\newcommand{\cS}{\mathcal{S}}
\newcommand{\N}{\mathbb{N}}

\def\S{\mathcal S}

\def\qed{\hfill $\Box$\medskip}

\newtheorem{theorem}{Theorem}[section]

\newtheorem{lemma}[theorem]{Lemma}
\newtheorem{example}[theorem]{Example}

\newtheorem{definition}[theorem]{Definition}

\title{Perfect partition of some regular bipartite graphs}
\date{\today}

\author{Chi-Kwong Li$^{1}$\and Jeff Soosiah$^{2}$\and Gexin Yu$^{3}$}

\thanks{$^{1}$ Research supported in part by NSF}
\thanks{$^{2}$ Research supported in part by NSF REU}
\thanks{$^{3}$ Research supported in part by the NSA grant H98230-12-1-0226.}
\address{Department of Mathematics, College of William and Mary, Williamsburg, VA 23187.}\email{ckli@math.wm.edu, jsoosiah@gmail.com, gyu@wm.edu}

\begin{document}
\maketitle

\begin{abstract}
A graph has a perfect partition if all its perfect matchings can be partitioned so that each part is a 1-factorization of the graph.  Let $L_{rm, r}=K_{rm,rm}-mK_{r,r}$.  
We first give a formula to count the number of perfect matchings of $L_{rm, r}$, then show that $L_{6,1}$ and $L_{8,2}$ have perfect partitions. 
\end{abstract}

\section{Introduction} 

It is well-known that every regular bipartite graph has a perfect matching, 
and furthermore, every perfect matching of a regular bipartite graph is in 
a $1$-factorization, that is, a collection of pairwise perfect matchings whose 
union is the original graph.  Here we study the inverse problem: given the perfect 
matchings of a graph, can one partition them into $1$-factorizations? 

This perfect partition problem was introduced in \cite{BCL05}, 
in the language of matrices.  Let $S_n$ be the set of $n\times n$ permutation 
matrices, i.e.,  $(0,1)$-matrix  each of whose row and column contains exactly one $1$.  Let $A$ be an $n\times n$ $(0,1)$-matrix, a permutation matrix $P$  is contained in $A$, denoted by $P<A$, if $A-P$ has nonnegative entries, 
and we let   $$\cS(A)=\{P\in S_n:  P<A \}.$$  
We say that a $(0,1)$-matrix has a perfect partition if $\cS(A)$ 
can be partitioned into subsets so that the sum 
of permutation matrices in each subset is $A$.  

One can see that the two definitions in the preceding paragraphs are equivalent.  
For a regular bipartite graph $G$, let the {\em associated matrix $A(G)$} of $G$ be  
the adjacency matrix so that the rows and columns  are the two parts of $G$.  
Thus a perfect matching in $G$ is a permutation matrix in the  associated matrix 
$A(G)$, and a $1$-factorization of $G$ is a set of permutation matrices 
whose sum is $A(G)$. 

It is not hard to construct bipartite graphs or square matrices which have no perfect 
partition. 

\begin{example} \rm Let $P$ be the permutation matrix corresponding to permutation $(12345)$ in the cycle representation, 
and let $A=I_5+P+P^2$. Then every $1$-factorization of the graph $G(A)$ associated with $A$ 
should contain exactly three perfect matchings, but $G(A)$ contains $13$ perfect matchings. 
\end{example}

The perfect partition problem is interesting for graphs in which every perfect matching is 
in a $1$-factorization.  Because of this, we consider a special kind of regular bipartite graphs 
which contain many perfect matchings. 

For $r, m\in \N$, let $L_{rm,r}=K_{rm, rm}-mK_{r,r}$.  
Denote by $J_n$ the $n\times n$ matrices with all entries equal to 1.
In terms of matrices, that is the matrix 
obtained from  $J_{rm\times rm}$  by replacing the ones on the  $r\times r$ disjoint  submatices 
on the main diagonal by zeros.  Evidently, a permutation matrix 
$P$ satisfies $P < A(L_{n,1})$ if and only if it is a derangement,
i.e., a permutation matrix with zero diagonal entries.

In ~\cite{BCL05}, the authors showed that the number of perfect matchings in $G$ equals to the 
permanent of the associated matrix $A(G)$, and that the permanent of $A(L_{rm,r})$ 
is a multiple of $r(m-1)$ using the Laplace expansion formula for permanent, see 
[\cite{BR}, p199].  Thus $L_{rm,r}$ satisfies the easy necessary condition to have a 
perfect partition.  Here we give a formula to calculate the number of perfect 
matchings in $L_{rm,r}$, which generalizes the formula for the number of derangements [\cite{BR}, p202]: $D_n=n!\sum_{k=0}^n (-1)^k \frac{1}{k!}$. 
 
\begin{theorem}\label{counting}
Let $n=rm$.  Then the number of perfect matchings in $L_{n,r}$ is 
$$|M(L_{n,r})| = a_0n! - a_1(n-1)! + \cdots + (-1)^na_n0!=n!\sum_{k=0}^n \frac{a_k}{k!},$$ 
where $$a_0 + a_1x+ a_2x^2 + \cdots + a_nx^n=\left(\sum_{k=0}^r k!{r\choose k}^2 x^k\right)^m$$ 
\end{theorem}

\begin{proof}
View $L_{n,r}$ as a chessboard with forbidden positions at zeroes in the matrix,  we can think of 
a permutation matrix (perfect matching) to be a way to $n$ non-attacking rooks so that none of the
rooks are in those forbidden positions.  Let $a_i$ be the number of ways to place $n$ non-attacking 
rooks so that at least $m$ rooks are in the forbidden positions, then by the Principle of Inclusion 
and Exclusion, 
$$|M(L_{n,r})|=\sum_{i=0}^n (-1)^i a_i (n-i)!.$$

Let $A(x)=\sum_{i}a_ix^i$ be the rook polynomial.  Then $A(x)=B(x)^m$, where 
$B(x)=\sum_{k=0}^r b_k  x^k$ is the rook polynomial to place non-attacking rooks in the 
$r\times r$ matrix, since the ways to  place rooks in the $r$ square matrices do not 
interfere with each other.   Thus we have the theorem. 
\end{proof}

For some special $m$ and $r$, it is not hard to show $L_{rm,r}$ has a perfect partition.  
For example,  let $G=K_{n,n}$ and $A(G)$ be the corresponding matrix of $G$.  
Then every permutation matrix of $A(G)$ corresponds to an element in $S_n$.
Let $P$ be a permutation of order $n$, then the subgroup $H$ generated 
by $P$ gives a $1$-factorization of $G$. The left cosets of $H$ give a perfect partition of $G$. 
So, we have the following, see~\cite{BCL05}.

\begin{theorem}
The complete bipartite graphs $K_{n,n}$ (or $L_{n,0}$) has a perfect partition. 
\end{theorem}
 
By using the cosets of subgroups, it is also shown in~\cite{BCL05} that the set $A_n$ of even permutations has a perfect partition and $L_{2n,n}$ has a perfect partition.  
However, it becomes hard to solve when the permutations in a given
matrix have no group structures.    
In \cite{BCL05}, the authors obtained the following.

\begin{theorem}
The graphs $L_{4,1}, L_{5,1}$ and $L_{6,2}$ have perfect partitions. 
\end{theorem}

It is challenging to find a perfect partition for $L_{n,1}$ when $n\ge 6$. 
Note that we need to find a perfect partition of the derangements of $n$ elements.  
In \cite{BCL05}, five different strategies were proposed to show  that $L_{6,1}$ has a perfect 
partition, but none of them led to a solution.   
In this paper, we use a different strategy to show that $L_{6,1}$ indeed has a perfect partition.  
It is easy enough to list a perfect partition of $L_{6,1}$ into 53 sets with 5 perfect matchings each.
Nevertheless, we will give a theoretical proof in Section 2, and hope that the
proof techniques can inspire future study on $L_{n,1}$ for $n \ge 7$. (Note that it is not so easy to list $792$ sets with $6$ perfect matchings in the case of $L_{7,1}$.)
Also, we will show that the graph $L_{8,2}$ has a perfect partition in Section 3.
Our construction for $L_{8,2}$ used the perfect partition of $L_{4,1}$ from \cite{BCL05}. 
Again, we hope that the techniques can inspire future advance of the partition problem.

\iffalse

Unfortunately, the construction for $L_{6,1}$ highly depends on the specific structure 
of $L_{6,1}$, thus it provides no idea on how to show $L_{n,1}$ has a perfect partition for $n>6$.

We are also able to show that $L_{8,2}$ has a perfect partition. 

Interestingly, our construction for $L_{8,2}$ used the perfect partition of $L_{4,1}$ 
from \cite{BCL05} as a reference.   

\fi

\section{perfect partition of $L_{6,1}$}

The purpose of this section is to prove the following.

\begin{theorem} \label{2.1}
The graph $L_{6,1}$ has a perfect partition.
\end{theorem}

We divide the proofs of Theorem \ref{2.1} into several lemmas. For notational 
convenience, we will not distinguish between the graph $L_{rm,r}$ and it 
adjacency matrix $A(L_{rm,r})$. 
 
By Theorem~\ref{counting}, $L_{6,1}$ consists of the 265 derangements in $S_6$.  
According to their cycle structure, we see the following.

\begin{lemma}
Let $T$ be the derangements with one $6$-cycle, $C_{33}$ be the derangements with a product of 
two $3$-cycles, $C_{24}$ be the derangements with a product of a $2$-cycle and $4$-cycle, and among
those,  $C_{24}^0$ be the ones with $1$ in the $2$-cycle, and let $C_{222}$ be the derangements 
with a product of three $2$-cycles.   Then $\cS(L_{6,1})=C_6\cup C_{24}\cup C_{33}\cup C_{222}$ and 
furthermore, 
$$|C_6|=120,  |C_{33}|=40,  |C_{24}|=90, |C_{24}^0|=30, |C_{222}|=15.$$
\end{lemma}

Our proof of Theorem \ref{2.1}  will use the following partition strategy: 

\begin{enumerate}[$T_1$:]
\item $30$ subsets consisting of four elements from $C_6$ and one element from $C_{24}^0$;
\item $16$ subsets consisting of three elements from $C_{24}-C_{24}^0$ and two elements from $C_{33}$;
\item $3$ subsets consisting of four elements from $C_{24}-C_{24}^0$ and one element from $C_{222}$;
\item $4$ subsets consisting of two elements from $C_{33}$ and three elements from $C_{222}$.
\end{enumerate}

%To be specific, we define elements in $T_1$ to be as following

For each $\sigma=(1\,x_2)(x_3\,x_4\,x_5\,x_6)\in C_{24}^0$, let $$f(\sigma)=\{\sigma, 
(1\,x_3\,x_2\,x_5\,x_4\,x_6),
(1\,x_4\,x_2\,x_6\,x_5\,x_3),
(1\,x_5\,x_2\,x_3\,x_6\,x_4),
(1\,x_6\,x_2\,x_4\,x_3\,x_5)\}.$$

\begin{lemma}
The set $T_1=\{f(\sigma): \sigma\in C_{24}^0\}$ gives a perfect partition of $C_{24}^0\cup C_6$.
 \end{lemma}

\begin{proof} It is easy to see that $f(\sigma)$ is a $1$-factorization for any give $\sigma\in C_6$. On the other hand,  given 
$\tau=(1\,y_2\,y_3\,y_4\,y_5\,y_6)\in C_6$,  we find that the only elements in $C_{2,4}^0$ whose images under $f$ contain $\tau$ are $(1\,y_3)(y_2\,y_5\,y_4\,y_6)$,  $(1\,y_3)(y_6\,y_2\,y_5\,y_4)$,  $(1\,y_3)(y_4\,y_6\,y_2\,y_5)$, $(1\,y_3)(y_5\,y_4\,y_6\,y_2)$, but these are all the same element. Thus each $C_6$ element is covered by exactly one $C_{2,4}^0$ element, so that the 30 elements of $C_{2,4}^0$ cover all 120 elements in $C_6$. 
\end{proof}

According to above strategy, here is the partition of $C_{24}^0\cup C_6$. 

\begin{table}[h]
\caption{30 sets in $T_1$, where each uses 4 from $C_6$ and 1 from $C_{24}^0$. } 
\centering 
\begin{tabular} { | r | c c c c c | c | }
\hline
1&(1 2)(3 4 5 6) &(1 3 2 5 4 6)  &(1 4 2 6 5 3)  &(1 5 2 3 6 4) &(1 6 2 4 3 5)\\\hline
2&(1 2)(3 4 6 5) &(1 3 2 6 4 5)  &(1 4 2 5 6 3)  &(1 6 2 3 5 4) &(1 5 2 4 3 6)\\\hline
3&(1 2)(3 5 6 4) &(1 3 2 6 5 4)  &(1 5 2 4 6 3)  &(1 6 2 3 4 5) &(1 4 2 5 3 6)\\\hline
4&(1 2)(3 5 4 6) &(1 3 2 4 5 6)  &(1 5 2 6 4 3)  &(1 4 2 3 6 5) &(1 6 2 5 3 4)\\\hline
5&(1 2)(3 6 5 4) &(1 3 2 5 6 4)  &(1 6 2 4 5 3)  &(1 5 2 3 4 6) &(1 4 2 6 3 5)\\\hline
6&(1 2)(3 6 4 5) &(1 3 2 4 6 5)  &(1 6 2 5 4 3)  &(1 4 2 3 5 6) &(1 5 2 6 3 4)\\\hline
7&(1 3)(2 4 5 6) &(1 2 3 5 4 6)  &(1 4 3 6 5 2)  &(1 5 3 2 6 4) &(1 6 3 4 2 5)\\\hline
8&(1 3)(2 4 6 5) &(1 2 3 6 4 5)  &(1 4 3 5 6 2)  &(1 6 3 2 5 4) &(1 5 3 4 2 6)\\\hline
9&(1 3)(2 5 6 4) &(1 2 3 6 5 4)  &(1 5 3 4 6 2)  &(1 6 3 2 4 5) &(1 4 3 5 2 6)\\\hline
10&(1 3)(2 5 4 6) &(1 2 3 4 5 6)  &(1 5 3 6 4 2)  &(1 4 3 2 6 5) &(1 6 3 5 2 4)\\\hline
11&(1 3)(2 6 5 4) &(1 2 3 5 6 4)  &(1 6 3 4 5 2)  &(1 5 3 2 4 6) &(1 4 3 6 2 5)\\\hline
12&(1 3)(2 6 4 5) &(1 2 3 4 6 5)  &(1 6 3 5 4 2)  &(1 4 3 2 5 6) &(1 5 3 6 2 4)\\\hline
13&(1 4)(2 3 5 6) &(1 2 4 5 3 6)  &(1 3 4 6 5 2)  &(1 5 4 2 6 3) &(1 6 4 3 2 5)\\\hline
14&(1 4)(2 3 6 5) &(1 2 4 6 3 5)  &(1 3 4 5 6 2)  &(1 6 4 2 5 3) &(1 5 4 3 2 6)\\\hline
15&(1 4)(2 5 6 3) &(1 2 4 6 5 3)  &(1 5 4 3 6 2)  &(1 6 4 2 3 5) &(1 3 4 5 2 6)\\\hline
16&(1 4)(2 5 3 6) &(1 2 4 3 5 6)  &(1 5 4 6 3 2)  &(1 3 4 2 6 5) &(1 6 4 5 2 3)\\\hline
17&(1 4)(2 6 5 3) &(1 2 4 5 6 3)  &(1 6 4 3 5 2)  &(1 5 4 2 3 6) &(1 3 4 6 2 5)\\\hline
18&(1 4)(2 6 3 5) &(1 2 4 3 6 5)  &(1 6 4 5 3 2)  &(1 3 4 2 5 6) &(1 5 4 6 2 3)\\\hline
19&(1 5)(2 3 4 6) &(1 2 5 4 3 6)  &(1 3 5 6 4 2)  &(1 4 5 2 6 3) &(1 6 5 3 2 4)\\\hline
20&(1 5)(2 3 6 4) &(1 2 5 6 3 4)  &(1 3 5 4 6 2)  &(1 6 5 2 4 3) &(1 4 5 3 2 6)\\\hline
21&(1 5)(2 4 6 3) &(1 2 5 6 4 3)  &(1 4 5 3 6 2)  &(1 6 5 2 3 4) &(1 3 5 4 2 6)\\\hline
22&(1 5)(2 4 3 6) &(1 2 5 3 4 6)  &(1 4 5 6 3 2)  &(1 3 5 2 6 4) &(1 6 5 4 2 3)\\\hline
23&(1 5)(2 6 4 3) &(1 2 5 4 6 3)  &(1 6 5 3 4 2)  &(1 4 5 2 3 6) &(1 3 5 6 2 4)\\\hline
24&(1 5)(2 6 3 4) &(1 2 5 3 6 4)  &(1 6 5 4 3 2)  &(1 3 5 2 4 6) &(1 4 5 6 2 3)\\\hline
25&(1 6)(2 3 4 5) &(1 2 6 4 3 5)  &(1 3 6 5 4 2)  &(1 4 6 2 5 3) &(1 5 6 3 2 4)\\\hline
26&(1 6)(2 3 5 4) &(1 2 6 5 3 4)  &(1 3 6 4 5 2)  &(1 5 6 2 4 3) &(1 4 6 3 2 5)\\\hline
27&(1 6)(2 4 5 3) &(1 2 6 5 4 3)  &(1 4 6 3 5 2)  &(1 5 6 2 3 4) &(1 3 6 4 2 5)\\\hline
28&(1 6)(2 4 3 5) &(1 2 6 3 4 5)  &(1 4 6 5 3 2)  &(1 3 6 2 5 4) &(1 5 6 4 2 3)\\\hline
29&(1 6)(2 5 4 3) &(1 2 6 4 5 3)  &(1 5 6 3 4 2)  &(1 4 6 2 3 5) &(1 3 6 5 2 4)\\\hline
30&(1 6)(2 5 3 4) &(1 2 6 3 5 4)  &(1 5 6 4 3 2)  &(1 3 6 2 4 5) &(1 4 6 5 2 3)\\\hline
\end{tabular}
\label{table:T1}
\end{table}

\begin{definition}\label{class}
Let $\sigma=(1\,x\,y)(a\,b\,c)\in C_{3,3}$ with $a<b,c$. Then {\em the class} of $\sigma$ is $y$ if $b<c$ and $x$ otherwise. 
\end{definition}

By the definition, a permutation and its inverse have the same class; furthermore, if they are  of class $y$, then one of the them can be written as $(1xy)(abc)$ so that $a<b<c$.   In the following of the paper, we will refer $\sigma$ to be $(1xy)(abc)$ and use $\sigma^{-1}$ to be the inverse. 

Note that if $\sigma=(1xy)(abc)$ has class $y$, then $\sigma^*=(1yx)(abc)$ and its inverse have class $x$.  

Now suppose $\sigma=(1\,x\,y)(a\,b\,c)$  has class $y$. We will choose the subset for $\sigma, \sigma^{-1}$ so that each $C_{2,4}-C_{2,4}^0$ element contains $y$ in the 2-cycle, 
%$$\{(1\,x\,y)(a\,b\,c), (1\,y\,x)(a\,c\,b), (1\,\cdot\,x\,\cdot\,)(y\,\cdot\,), (1\,\cdot\,x\,\cdot\,)(y\,\cdot\,), (1\,\cdot\,x\,\cdot\,)(y\,\cdot\,)\}$$
then there are two possible ways to finish determining the $C_{2,4}-C_{2,4}^0$ elements:
\begin{eqnarray}
\{(1\,x\,y)(a\,b\,c), (1\,y\,x)(a\,c\,b), (1\,{\underline a}\,x\,{\underline b}\,)(y\,{\underline c}), (1\,{\underline c}\,x\,{\underline a}\,)(y\,{\underline b}), (1\,{\underline b}\,x\,{\underline c}\,)(y\,{\underline a})\}\label{pattern_abc}\\
\{(1\,x\,y)(a\,b\,c), 
(1\,y\,x)(a\,c\,b), 
(1\,{\underline a}\,x\,{\underline c}\,)(y\,{\underline b}), 
(1\,{\underline b}\,x\,{\underline a}\,)(y\,{\underline c}), 
(1\,{\underline c}\,x\,{\underline b}\,)(y\,{\underline a})\}\label{pattern_acb}
\end{eqnarray}

We introduce the concept of a {\it pattern} to decide the associated subsets with $\sigma\in C_{3,3}$. 

\begin{definition}
Let $\sigma=(1\,x\,y)(a\,b\,c)$ be of class $y$. Then $\sigma$ and $\sigma^{-1}$ have pattern $\beta=(1xy)(wvu)$ where $\{w, v, u\}=\{a, b, c\}$ if and only if $\sigma$ and $\sigma^{-1}$ are associated with the following three elements in $C_{2,4}$:
$$(1wxv)(yu), (1vxu)(yw), (1uxw)(yv).$$  
We will let the pattern for $\sigma^*$ and $(\sigma^*)^{-1}$ be $\beta^{-1}$. 
\end{definition}

For example, if $\sigma=(1xy)(abc)$ of class $y$ has pattern $(1xy)(abc)$, then $\sigma, \sigma^{-1}$ and their associated elements give \eqref{pattern_abc}, and $\sigma^*, (\sigma^*)^{-1}$ and their associated elements give \eqref{pattern_acb}.  

%So the patter for $\sigma=(1xy)(abc)$ could be $(1xy)(abc)$ or $(1xy)(acb)$. 

%\begin{lemma}
%Let $\sigma=(1xy)(abc)$ be of class $y$. Then \\
%(i)  $\sigma$ has pattern $\beta=(1xy)(***)$;\\
%(ii) $\sigma^{-1}$ has the same pattern as $\sigma$;\\
%(iii) $\sigma^*=(1xy)(acb)$ and its inverse have patter $\beta^{-1}$.
%\end{lemma}

%\begin{proof}
%well-defined. 
%\end{proof}

For a given element $\sigma=(1xy)(abc)\in C_{33}$ of class $y$ with pattern $\beta=(1xy)(ab'c')$, where $\{b',c'\}=\{b,c\}$, we can define the set $\mathcal{Z}_y^{\beta}(\sigma)$, {\em the zone $y$} which consists of four subsets of five permutations, according to the following rules:

\begin{enumerate}[(i)]
\item Determine the three elements in $C_{2,4}-C_{2,4}^0$ associated with $\sigma$ and $\sigma^{-1}$;
\item Determine the other three pairs of elements in $C_{3,3}$ with class $y$. By definition, they have the form $\gamma=(1*y)(***)$ (and $\gamma^{-1}$) so that the three elements in the second cycle are in increasing order. 
\item Determine the pattern for each $\gamma$ and $\gamma^{-1}$:  if $\gamma$ is associated with $(1uvw)(yk)\in C_{2,4}$ and $\sigma$ is associated with $(1u'xw')(yk)\in C_{2,4}$, then $(u'xw')=(uvw)$ but $u'xw'\not=uvw$. 
\item Write down the elements in $\mathcal{Z}_y^{\beta}(\sigma)$, which are the four pairs of class $y$ elements together with their associated elements in $C_{2,4}-C_{2,4}^0$. 
\end{enumerate}

For example, if we take $\sigma=(123)(465)$ with pattern $(132)(465)$, then we will get zone $2$ (note that $(123)(465)$ is of class $2$) as follows: \\

\begin{tabular} { | c c| c c c c | c | }
\hline
 (1 2 3)(4 6 5) &(1 3 2)(4 5 6) &  & (1 6 3 5)(2 4)&   (1 4 3 6)(2 5) &(1 5 3 4)(2 6)\\\hline
 (1 2 4)(3 6 5) &(1 4 2)(3 5 6) &(1 5 4 6)(2 3)&&(1 6 4 3)(2 5)&(1 3 4 5)(2 6) \\\hline 
 (1 2 5)(3 6 4) &(1 5 2)(3 4 6) &(1 6 5 4)(2 3)&(1 3 5 6)(2 4) &&(1 4 5 3)(2 6)\\\hline
 (1 2 6)(3 5 4) &(1 6 2)(3 4 5) &(1 4 6 5)(2 3)&(1 5 6 3)(2 4)&(1 3 6 4)(2 5) &\\\hline
 \end{tabular}
 
 \bigskip

\begin{lemma}
For a given $\sigma\in C_{3,3}$ of class $y>1$ with a given pattern $\beta$, $\mathcal{Z}_y^{\beta}$ consists of four disjoint subsets of five distinct permutations.  Furthermore, if $\sigma, \gamma\in C_{3,3}$ are two elements in the subsets of $\mathcal{Z}_y$ with patterns $\beta_{\sigma}$ and $\beta_{\gamma}$, respectively, then $\mathcal{Z}_y^{\beta_{\sigma}}(\sigma)=\mathcal{Z}_y^{\beta_{\gamma}}(\gamma)$.\end{lemma}

\begin{proof}
By definition, the three elements in $C_{3,3}$ are determined. Also, once we know the patterns, then the associated elements in $C_{2,4}$ are also determined.  So we only need to show that the patterns are well-defined as well.  Let $(1uxv)(yk), (1vxk)(yu)\in C_{2,4}$ be associated with $\sigma$ which are used to determine the pattern of $\gamma$.  First, $\gamma=(1vy)(***)$.   So the corresponding associated elements from $C_{2,4}$ are $(1xvu)(yk)$ and $(1kvx)(yu)$, and the patterns are $(1vy)(xuk)$ and $(1vy)(kxu)$ which are the same.

For the ``furthermore'' part, we just need to show that $\sigma$ with pattern $\beta_{\sigma}$ determines $\gamma$ and its pattern $\beta_{\gamma}$, then the converse is also true. 
One can readily verify this statement.
\end{proof}

%If $\sigma=(1xy)(abc)$ has pattern $\simga'$, then $\sigma^*=(1xy)(acb)$ has pattern $\sigma'^{-1}$.  

\begin{lemma}
For each $z\in \{2, 3, 4, 5, 6\}-\{y\}$, $\mathcal{Z}_y^{\beta}(\sigma)$ uniquely determines sets $\mathcal{Z}_z$ so that $\mathcal{Z}_z\cap \mathcal{Z}_y^{\beta}(\sigma)=\emptyset$. Moreover, if $z\not=z'$, then $\mathcal{Z}_z\cap \mathcal{Z}_{z'}=\emptyset$. 
\end{lemma}

\begin{proof}
Let $\sigma=(1zy)(abc)\in C_{3,3}$ be in zone $y$ with pattern $\beta$, then $\sigma^*=(1zy)(acb)$ of class $z$ is in zone $z$ with pattern $\beta^{-1}$, thus $\mathcal{Z}_z=\mathcal{Z}_z^{\beta^{-1}}(\sigma^*)$, by the process described above.   By construction and the previous Lemma, the set $\mathcal{Z}_z$ is unique. 

Similarly, suppose that $(1z'z)(abc)\in \mathcal{Z}_z$ is of pattern $\beta_1$, then $(1zz')(abc)\in \mathcal{Z}_{z'}$ is of pattern $\beta_1^{-1}$, so $\mathcal{Z}_z$ determines $\mathcal{Z}_{z'}$. 
\end{proof}

For example, for the previous $\sigma$ and the pattern, we could get zone $3$ as follows:

\bigskip

 \begin{tabular} { | r | c c c c c | c | }
\hline
(1 2 3)(4 5 6) &(1 3 2)(4 6 5) & &(1 5 2 6)(3 4)&(1 6 2 4)(3 5)&(1 4 2 5)(3 6)\\\hline
(1 3 4)(2 6 5) &(1 4 3)(2 5 6) &(1 6 4 5)(3 2)&&(1 2 4 6)(3 5) &(1 5 4 2)(3 6)\\\hline
(1 3 5)(2 6 4) &(1 5 3)(2 4 6)& (1 4 5 6)(32)&(1 6 5 2)(3 4)&&(1 2 5 4)(3 6)\\\hline
(1 3 6)(2 5 4) &(1 6 3)(2 4 5) &(1 5 6 4)(3 2)&(1 2 6 5)(3 4) &(1 4 6 2)(3 5)&\\\hline
\end{tabular}

\bigskip

The following is the zone $4$: 
\bigskip

\begin{tabular} { | r | c c c c c | c | }\hline
(1 2 4)(3 5 6) &(1 4 2)(3 6 5) &&(1 6 2 5)(4 3)&(1 3 2 6)(4 5) &(1 5 2 3)(4 6)\\\hline
(1 3 4)(2 5 6) &(1 4 3)(2 6 5) &(1 5 3 6)(4 2)&&(1 6 3 2)(4 5)&(1 2 3 5)(4 6) \\\hline
(1 4 5)(2 6 3) &(1 5 4)(2 3 6) &(1 6 5 3)(4 2)&(1 2 5 6)(4 3) &&(1 3 5 2)(4 6)\\\hline
(1 4 6)(2 5 3) &(1 6 4)(2 3 5) &(1 3 6 5)(4 2)&(1 5 6 2)(4 3)&(1 2 6 3)(4 5) &\\\hline
\end{tabular}
\bigskip

Here is zone $5$:
\bigskip

\begin{tabular} { |  c c c c c  c | }\hline
(1 2 5)(3 4 6) &(1 5 2)(3 6 4) &&(1 4 2 6)(3 5)&(1 6 2 3)(4 5) &(1 3 2 4)(6 5)\\\hline
(1 3 5)(2 4 6) &(1 5 3)(2 6 4) &(1 6 3 4)(2 5)&&(1 2 3 6)(4 5)&(1 4 3 2)(6 5) \\\hline
(1 4 5)(2 3 6) &(1 5 4)(2 6 3) &(1 3 4 6)(2 5)&(1 6 4 2)(3 5) &&(1 2 4 3)(6 5)\\\hline
(1 6 5)(2 4 3) &(1 5 6)(2 3 4) &(1 4 6 3)(2 5)&(1 2 6 4)(3 5)&(1 3 6 2)(4 5)& \\\hline
\end{tabular}
\bigskip

Now zone $6$:

\bigskip

\begin{tabular} { |  c c c c c  c | }\hline
(1 2 6)(3 4 5) &(1 6 2)(3 5 4) &&(1 5 2 4)(3 6)&(1 3 2 5)(4 6) &(1 4 2 3)(5 6)\\\hline
(1 3 6)(2 4 5) &(1 6 3)(2 5 4) &(1 4 3 5)(2 6)&&(1 5 3 2)(4 6)&(1 2 3 4)(5 6) \\\hline
(1 4 6)(2 3 5) &(1 6 4)(2 5 3) &(1 5 4 3)(2 6)&(1 2 4 5)(3 6) &&(1 3 4 2)(5 6)\\\hline
(1 5 6)(2 3 4) &(1 6 5)(2 4 3) &(1 3 5 4)(2 6)&(1 4 5 2)(3 6)&(1 2 5 3)(4 6)& \\\hline
\end{tabular}

\bigskip

For a given $y_0\in \{2, 3, 4, 5, 6\}$,  the elements in $\cup_{y\in [6]-\{1,y_0\}}\mathcal{Z}_y$ form $T_2$.   We will take the elements in $\mathcal{Z}_y$ together with elements in $C_{2,2,2}$  to form $T_3$ and $T_4$. 

For $\sigma=(1xy_0)(a'b'c')\in C_{3,3}$ of class $y_0$ with pattern $(1xy_0)(abc)$, we let 
$$f(\sigma)=\{\sigma, \sigma^{-1},  (1a)(xb)(y_0c),  (1b)(xc)(y_0a),  (1c)(xa)(y_0b)\}.$$

\begin{lemma}
The set $T_4=\{f(\sigma): \sigma\in C_{3,3} \text{ and of class $y_0$}\}$ is a perfect partition of class $y_0$ elements in $C_{3,3}$ and elements in $C_{2,2,2}$ with no $2$-cycle $(1y_0)$. 
\end{lemma}

\begin{proof}
We just need to show that $f(\sigma)\cap f(\gamma)=\emptyset$ if $\sigma\not=\gamma$.  Suppose that $(1a)(xb)(y_0c)\in f(\sigma)\cap f(\gamma)$.  Then $\sigma$ has pattern $(1xy_0)(abc)$ and $\gamma$ has pattern $(1by_0)(axc)$.  Therefore the elements in $C_{2,4}$ associated with $\sigma$ and $\gamma$ are $(1axb)(y_0c), (1bxc)(y_0a), (1cxa)(y_0b)$ and $(1abx)(y_0c), (1xbc)(y_0a), (1cba)(y_0x)$, respectively.   But then we have $(1bxc)(y_0a)$ and $(1xbc)(y_0a)$ in the lists, which is a contradiction to a property of $Z_{y_0}$. 
\end{proof}

For example, if let $y_0=5$, then we have $T_4$ as follows:

\begin{table}[ht]
%\caption{4 sets in $T_4$, where each uses 2 from $C_{33}$ and 3 from $C_{222}$. } 
\centering 
\begin{tabular} { |  c c c c c  | }
\hline
 (12)(43)(56)&(13)(46)(52)&(16)(42)(53)&(1 4 5)(2 3 6)&(1 5 4)(2 6 3)\\\hline
 (12)(64)(53)&(13)(62)(54)&(14)(63)(25)&(1 5 6)(2 3 4)&(1 6 5)(2 4 3)\\\hline
 (12)(36)(54)&(14)(32)(65)&(16)(34)(52)&(1 3 5)(2 4 6)&(1 5 3)(2 6 4)\\\hline
 (13)(2 4)(56)&(14)(26)(53)&(16)(23)(54)&(1 2 5)(3 4 6)&(1 5 2)(3 6 4)\\\hline
\end{tabular}
\label{table:T4}
\end{table} 

Now we define $T_3$.   For $\mu=(1y_0)(xa')(b'c')\in C_{2,2,2}$, we let 
$$f(\mu)=\{\mu, (1b'a'c')(y_0x), (1c'xb')(y_0a'), (1xc'a')(y_0b'), (1a'b'x)(y_0c')\},$$
where $(b'a'c')=(abc)$. 

\begin{lemma}
The set $T_3=\{f(\mu): \mu=(1y_0)(**)(**)\in C_{2,2,2}\}$ is a perfect partition of elements in $C_{2,2,2}$ with a $2$-cycle $(1y_0)$ and the elements in $C_{2,4}$ in $\mathcal{Z}_{y_0}^{\beta}$, where $\beta=(1xy_0)(abc)$.
\end{lemma}

\begin{proof}
We just need to show that $f(\mu)\cap f(\rho)=\emptyset$ if $\mu\not=\rho$.  But $(1u'v'w')(y_0x')\in f(\mu)\cap f(\rho)$ only if $\mu=(1y_0)(x'v')(u'w')=\rho$. 
 \end{proof}
 
 So with the chosen $\sigma$ and the pattern, and $y_0=5$, we have $T_3$ as follows:
 \bigskip
%\begin{table}[h]
%\caption{3 sets in $T_3$, where each uses 4 from $C_{24}$ and 1 from $C_{222}$. } 
\begin{center}\begin{tabular} { |  c c c c c | }
\hline\hline
 (1 5)(2 3)(4 6)&(1 2 4 3)(5 6) &(1 3 6 2)(5 4) &(1 4 2 6)(5 3)&(1 6 3 4)(5 2)\\\hline
 (1 5)(2 4)(3 6)&(1 2 6 4)(5 3) &(1 3 4 6)(5 2) &(1 4 3 2)(5 6)&(1 6 2 3)(5  4)\\\hline
 (1 5)(2 6)(3 4)&(1 2 3 6)(5 4) &(1 3 2 4)(5 6) &(1 4 6 3)(5 2)&(1 6 4 2)(5 3)\\\hline
\end{tabular}\end{center}
%\label{table:T3}
%\end{table}
\bigskip 
 
By the lemmas, we have constructed a perfect partition of $L_{6,1}$, and the conclusion of
Theorem \ref{2.1} follows.

\section{perfect partition $L_{8,2}$}

The main theorem of this section is the following.

\begin{theorem} \label{3.1}
The graph $L\{8,2)$ has a perfect partition. 
\end{theorem}

\noindent
We will use the following notation.

$M_n$: the set of $n\times n$ real matrices,

$\{E_{11}, E_{12}, \dots, E_{nn}\}$: standard basis for $M_n$,

$J_n \in M_n$: the matrix with all entries equal to one,

$O_n\in M_n$: the matrix with all entries equal to zero,

%$I_k \oplus J_m = J_m\oplus \cdots \oplus J_m$, $k$ copies,

%$L_{8,2} = J_8 - (I_4 \oplus J_2)$, 

$C(i,j)$: Swap columns $i$ and $j$ in matrix,

$R(i,j)$: Swap rows $i$ and $j$ in matrix.

\iffalse
$\S(A)$:
the set of permutation matrices $P$ in $M_n$ such that 
$A-P$ is nonnegative for a given zero one matrix $A \in M_n$.
\fi

\begin{lemma}
Suppose a matrix $P \in \S(L_{8,2})$ is witten in 
block form $P = (P_{ij})_{1 \le i,j \le 4}$ so that
$P_{ij} \in M_2$ for every pair $(i,j)$.
Then either none, one, two, or four of the $P_{ij}$ blocks
are invertible, i.e., two of the four entries equal to 1.
Thus, $\S(L_{8,2})$ can be partitioned into 
$\S_0 \cup \S_1 \cup \S_2 \cup \S_4$,
where $\S_k$ consists of matrices 
$P = (P_{ij})_{1 \le i,j\le 4}$
in $\S(L_{8,2})$ 
such that exactly $k$ of the submatrices $P_{ij}$ 
are invertible.
Moreover, we have:
$$|\S_0| = 2^8 9, \quad
|\S_1| = 2^9 3, \quad
|\S_2| = 2^8 3, \quad
|\S_4| = 2^4 9.$$
\end{lemma}

\begin{proof}
The set $\S_0$ contains the matrices for which no blocks are invertible. 
Then all blocks contain no more than one 1, and each row and column of blocks contain exactly 
two blocks containing exactly one 1. Denote such a block by $E$. For the first block column, there 
are ${3\choose2}=3$ possible choices for which blocks are $E$. This selection determines that the 
block not chosen must be $O_2$, so the other two non-diagonal entries in that row must be $E$. 
The first block row also allows ${3\choose2}=3$ possible choices for which blocks are $X$, then 
the other selections of $E$ are determined. Thus, there are $3\cdot3=9$ combinations of $E$. 
Each $E$ may one of $E_{11}, E_{12}, E_{21}, E_{22}$. There are 4 ways to select two $E$ blocks, 
2 ways to select the next  four $E$ blocks, and 1 way to select the last two. Therefore, 
$|\S_0|=2^8 9$. 

The set $\S_1$ contains the matrices for which exactly one block is invertible. In the $2\times2$ 
case, this is true only when a block is $I_2$ or $R_2$. Denote such a matrix by $X$. Then there are 12 non-diagonal positions for which the first $X$ may be placed. All other blocks $E$ and $O_2$ are 
determined. The single $X$ may be $I_2$ or $R_2$, so it can be chosen in 2 ways. One $E$ block can 
be chosen in 4 ways, the next four $E$ blocks can be each chosen in 2 ways, and the last is determined.
Thus, $|\S_1|=2^8 12=2^9 3$. 

The set $\S_2$ contains the matrices for which exactly two blocks are invertible. Denote such blocks 
by $X$. Then the first $X$ can be placed in one of 12 non-diagonal positions. This placement allows
only 2 ways to choose the other $X$. Since order does not matter, we have $\frac{12\cdot2}{2}=12$ 
ways to choose the placement of two $X$ blocks. There are two choices for each $X$ block, 4 choices 
for the first $E$ block, 2 choices for the next two $E$ blocks, and 1 choice for the last $E$ block. Thus, $|\S_2|=2^6 12=2^8 3$.

The set $\S_4$ contains the matrices for which exactly four blocks are invertible. Denote those four 
by $X$; then all other blocks must be $O_2$. There are 3 ways to place one $X$ in the first block column. Then find the column whose diagonal position is in the same row as the $X$ in the first 
column. There are 3 ways to place one $X$ in this column. All other $X$ blocks are then determined, 
so there are $3\cdot3=9$  ways to place the $X$ blocks. There are 2 ways to choose each $X$, so 
$|\S_4|=2^4 9$.
\end{proof}

\medskip\noindent
{\bf Proof of  Theorem \ref{3.1}.}

We will use the following partitioning scheme for $L_{8,2}$:

\medskip\noindent
\quad {\bf Type I.} Pick two matrices from $\S_0$ and four matrices from $\S_1$ to form subsets.

\medskip\noindent
~~~{\bf Type II.} Pick four matrices from $\S_0$ and two matrices from $\S_2$ to form subsets.

\medskip\noindent
{\bf Type III.} Pick six matrices from $\S_4$ to form subsets.

\medskip
{\bf Pick $2$ matrices from $\S_0$ and $4$ matrices from $\S_1$ to form a subset:}   
In the block form, $P$ has $9$ perfect matchings which form a perfect partition of $3$ subsets of 
$L_{4,1}$:  
\begin{eqnarray*}
\{(1,2)(3,4), (1,3,2,4), (1,4,2,3)\},  \\
\{(1,3)(2,4), (1,2,3,4), (1,4,3,2)\}, \\
\{(1,4)(2,3), (1,2,4,3), (1,3,4,2)\}.
\end{eqnarray*}

%The life cycle is illustrated in Figure~\ref{fig:flowchart}.
Let $\S_0^1\subset \S_0$ be the subset containing all matrices whose non-diagonal zero blocks 
form a perfect matching of the form $(1,2)(3,4), (1,3)(2,4)$, or $(1,4)(2,3)$.

\begin{figure}[ht]
\centering
\includegraphics[scale=0.19]{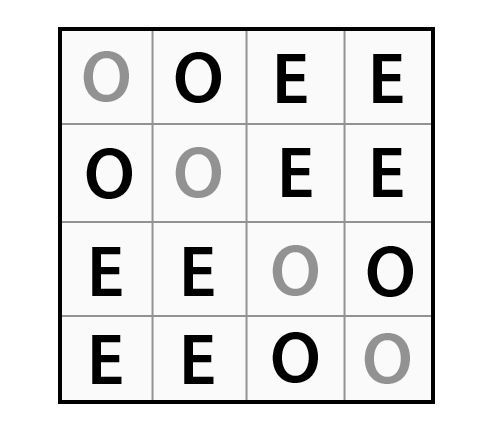}
\includegraphics[scale=0.19]{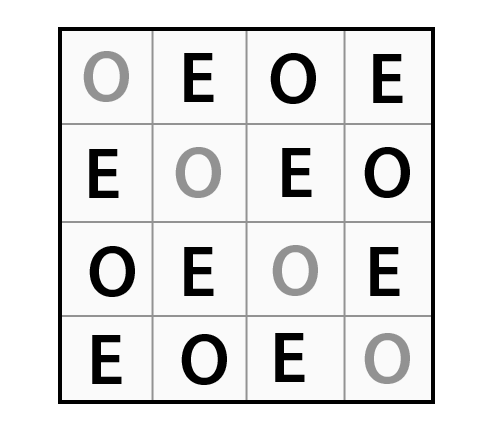}
\includegraphics[scale=0.19]{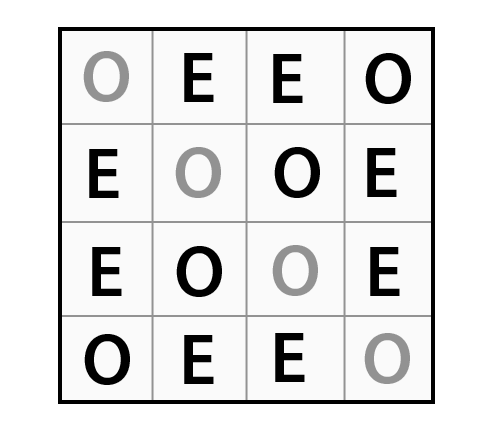}\\
\caption{Forms of the elements of $\S_0^1$: $(1,2)(3,4), (1,3)(2,4)$, or $(1,4)(2,3)$. }
\label{fig:s01}
\end{figure}

We will use $2$ matrices $A_1, A_2\in \S_0^1$ and $4$ matrices $S, T, U, V\in\S_1$ to form the partitions. 

Consider  a block form $(1,i)(j,k)$ now.  Take matrices $P, Q\in \S_0^1$ so that in the sum $A=P+Q$,  
the blocks $A_{1i}, A_{i1}, A_{jk}, A_{kj}$ are all $O_2$, and all other non-diagonal blocks are invertible (here is another way to describe it: choose 
$E_1=P_{1j}, E_2=P_{ik}, E_3=P_{j1}, E_4=P_{ki}$ freely, then $P$ and $Q$ are 
determined to have the desired $A$).

We shall choose $S,T, U, V\in\S_1$ so that their only invertible blocks are 
$S_{1i}, T_{i1},U_{jk}, V_{kj}$ (but we do not know them yet).  It follows that  
the blocks in the same rows and columns as the invertible ones are all $0$s.  
We also see that $S_{i1}, T_{1i}, U_{kj}, V_{jk}$ must be all $0$s, for otherwise, 
another block must be invertible as well.  

Since $A_{1j}$ is invertible and $S_{1j}=V_{1j}=O_2$, $T_{1j}+U_{1j}=J_2-A_{1j}$; and similarly,
$V_{j1}+S_{j1}=J_2-A_{j1}$.  We determine $T_{1j}$ by letting its $1$ to be in the same row as 
$P_{1k}$ and $V_{j1}$ by letting its $1$ to be in the same row as $P_{ji}$. 
Then, the following blocks are determined, based on the the additional fact that 
$S+T+U+V=L_{8,2}-A$:
\begin{equation}\label{det1}
T_{1j}\leftrightarrow T_{1k}\leftrightarrow V_{1k}\leftrightarrow V_{ik}\leftrightarrow 
S_{ik}\leftrightarrow S_{ij}\leftrightarrow U_{ij}\leftrightarrow U_{1j}\leftrightarrow T_{1j}
\end{equation}
\begin{equation}\label{det2}
V_{j1}\leftrightarrow V_{ji}\leftrightarrow T_{ji}\leftrightarrow T_{ki}\leftrightarrow 
U_{ki}\leftrightarrow U_{k1}\leftrightarrow S_{k1}\leftrightarrow S_{j1}\leftrightarrow V_{j1}
\end{equation} 

%Now consider the block $H_1=S_{jk}$, which is determined by the cycles given by (\ref{det1}) and (\ref{det2}) (specifically, by $S_{ik}$ and $S_{j1}$). Then $S_{jk}$, together with the invertible block $U_{jk}$, determines $T_{jk}$: if $S_{jk}=E_{a,b}$, then $T_{jk}=E_{3-a,3-b}$. 

Next, $S_{jk}$ is determined, as its $1$ is in different row from $S_{j1}$ (which is determined in \eqref{det2}) and in different column from $S_{ik}$ (which is determined in \eqref{det1}).  Then 
$T_{jk}$ and $U_{jk}$ are determined: if $S_{jk}=E_{a,b}$, then $T_{jk}=E_{3-a, 3-b}$ and 
$U_{jk}=J_2-(S_{jk}+T_{jk})$. 

Similarly, we can determine $S_{kj}, T_{kj}, V_{kj}$; $V_{1i}, U_{1i}, S_{1i}$; 
and $V_{i1}, U_{i1},  T_{i1}$.  At the end, we will get the matrices $S,T,U,V$.  

Note that we may get $S, T, U, V$ in a similar way by considering $T_{jk}$ first, then $S_{jk}$ and 
$U_{jk}$.  But it would be the same, since by the chain of determination in \eqref{det1} and 
\eqref{det2}, once we know one block in each chain, we know all other blocks, and since we determine 
$T_{jk}$ or $S_{jk}$ by choosing one block from \eqref{det1} and \eqref{det2},  there is no way we 
could get difference results.

%We show that $S_{jk}$ with (\ref{det1}) and (\ref{det2}) determine $T_{jk}$ in the same way. Given $S_{jk}=E_{a,b}$:

%\begin{enumerate}[(i)]
%\item By comparing rows, we have $S_{j1}=E_{3-a,n}; V_{j1}=E_{a,n}; V_{ji}=E_{3-a,n}; T_{ji}=E_{a,n};T_{jk}=E_{3-a,n}$. 
%\item By comparing columns, we have $S_{ik}=E_{n,3-b}; V_{ik}=E_{n,b}; V_{1k}=E_{n,3-b}; T_{1k}=E_{n,b}; T_{jk}=E_{n,3-b}$. 
%\end{enumerate}

%Thus $T_{jk}=E_{3-a,3-b}$ as determined by $S_{jk}$ and by cycles (\ref{det1}), (\ref{det2}). Symmetric results follow for all remaining singular blocks of $S,T,U,V$. 

%Now we show that each matrix in $\S_0^1\cup \S_1$ is included in exactly one such partition. 

From the above process, once we have chosen the pattern for the zero blocks and 
$P_{1j}, P_{j1}, P_{ik}$, and $P_{ki}$,  all six matrices are uniquely determined.  
By symmetry of $P$ and $Q$, there are 
$3\cdot (4\cdot 4\cdot 4\cdot 4\cdot \frac{1}{2})=3\cdot 2^7$ choices of different pairs $\{P,Q\}$.  
For each pair $\{P,Q\}$ in $\S_0^1$,  four different elements of $\S_1$ are used, so we actually 
use up all elements in $\S_1$. 

\bigskip

{\bf Pick $4$ matrices from $\S_0-\S_0^1$ and $2$ matrices from $\S_2$ to form a subset:} In the block 
form, $P$ has nine perfect matchings which forms perfect partition of three subsets. In partitions for 
$\S_0^1$ and $\S_1$, three perfect matchings were used. This partition will use the remaining six. 
Therefore, no element of $\S_0^1$ (that was used in the above partition) will be used in this 
partition. 

For a perfect matching $(1, i, j, k)$, (whose inverse is $(1, k, j, i)$), we choose a pair 
$A_1,  A_1'\in \S_0-\S_0^1$ so that
\begin{enumerate}[(i)]
\item $A_1$ and $A_1'$ have the same blocks at $(1, j), (j, 1), (k, i), (i, k)$; 
\item the blocks of $A_1$ at $(1, i), (i,j), (j,k), (k,1)$ are $O_2$; 
\item  the blocks of $A_1'$ at $(1, k), (k,j), (j,i), (i,1)$ are $O_2$. 
\end{enumerate}

We get $A_2$ so that $A_1+A_2$ only has zero blocks or invertible blocks,  and get $A_3, A_4$ from 
$A_2$ by applying operations 
$\{R(1,2),R(3,4),R(5,6),R(7,8)\}$ and $\{C(1,2),C(3,4),C(5,6),C(7,8)\}$, respectively.

Let $B_1, B_2\in \S_2$ so that $B_1+B_2=L_{8,2}-\sum_{i=1}^4 A_i$.  Because of the structure of 
matrices in $\S_2$,  $B_1$ and $B_2$ are determined if the sum is known.  Note that $S_{jk}$ is also 
determined by $S_{ij}$ and $S_{k1}$
Similarly we get $A_2', A_3', A_4'$ and $B_1', B_2'$ from $A_1'$. 

Now we show that every matrix in $\S_2\cup (\S_0-\S_0^1)$ appears exactly once in the above 
construction.  Note that for each choices of blocks at positions $(1,j), (j, 1), (k,i), (i,k)$, 
we get two different partitions of $L_{8,2}$,  with eight matrices in $\S_0$ and four in $\S_2$.   
We can partition the matrices in $\S_0-\S_0^1$ into sets of eight matrices, and each set uses four 
matrices in $\S_2$.  So in total $\frac{1}{2}\cdot |\S_0-\S_0^1|=3\cdot 2^8$ matrices in $\S_2$ 
are used, that is, we use up all matrices in $\S_2$. 

\bigskip

{\bf The perfect partitions using only matrices in $\S_4$:}  In the block form,  $P$ has $9$ 
perfect matchings which form a perfect partition with three subsets.    Each block $J_2$ in 
the perfect matchings can be decomposed into two invertible submatrices $I_2$ and $R_2$, so we 
will have six matrices from $\S_4$ summing to $L_{8,2}$.  

For each of the subsets,  we can apply one of the $7$ operations $\{R(1,2)\}$,  $\{R(3,4)\}$, 
$\{R(5,6)\}$, $\{R(7,8)\}$ $\{R(1,2), R(3,4)\}$,  $\{R(1,2), R(5,6)\}$, $\{R(1,2), R(7,8)\}$ 
to get a different partition.  In such a way, we use up all $18\cdot 8=9\cdot 2^4$ matrics in 
$\S_4$ to form perfect partitions.  
\qed

%\section*{Appendix: A partition of $L_{6,1}$}

\end{document}